\documentclass[preprint,12pt,3p]{elsarticle}
\usepackage{verbatim}
\usepackage{amsmath,amsthm,amsfonts,amssymb}
\usepackage{algorithm}
\usepackage{algpseudocode}
\usepackage{cases}
\usepackage{hyperref}
\usepackage[normalem]{ulem}
\usepackage{xcolor,sectsty}

\usepackage{hyperref}

\definecolor{astral}{RGB}{46,116,181}
\subsectionfont{\color{astral}}
\sectionfont{\color{astral}}
\newtheorem{theorem}{Theorem}[section]
\newtheorem{lemma}[theorem]{Lemma}

\newtheorem{definition}[theorem]{Definition}

\definecolor{darkslategray}{rgb}{0.18, 0.31, 0.31}
\definecolor{warmblack}{rgb}{0.0, 0.26, 0.26}
\journal{arXiv.org}

\newcommand{\C}{{\mathbb C}}

\newcommand{\mc}[1]{\mathcal {#1}}
\newcommand{\dg}{{\dagger}}
\newcommand{\n}{{*_{N}}}
\newcommand{\m}{{*_{M}}}
\newcommand{\kp}{{*_{K}}}
\newcommand{\lp}{{*_{L}}}

\begin{document}

\begin{frontmatter}

\title{ \textcolor{warmblack}{\bf Reverse-order law for weighted Moore--Penrose inverse of tensors}}

\author{Krushnachandra Panigrahy$^\dag$$^a$, Debasisha Mishra$^\dag$$^b$}

\address{               $^{\dag}$Department of Mathematics,\\
                        National Institute of Technology Raipur,\\
                        Raipur, Chhattisgarh, India.\\
                        \textit{E-mail$^a$}: \texttt{kcp.224\symbol{'100}gmail.com }\\
                        \textit{E-mail$^b$}: \texttt{dmishra\symbol{'100}nitrr.ac.in. }
        }

\begin{abstract}
\textcolor{warmblack}{
In this paper, we provide a few properties of the weighted Moore--Penrose inverse for an arbitrary order tensor via the Einstein product. We again obtain some new sufficient conditions for the reverse-order law of the weighted  Moore--Penrose inverse for even-order square tensors. We then present several characterizations of the reverse-order law for tensors of arbitrary order.}

\end{abstract}

\begin{keyword}
Tensor  \sep Moore--Penrose inverse \sep Weighted Moore--Penrose inverse \sep Einstein product \sep Reverse-order law.
\end{keyword}

\end{frontmatter}

\section{Introduction}\label{sec1}
A {\it tensor} is a multidimensional array. An element of ${\C}^{I_{1}\times\cdots\times I_{N}}$  is an  {\it $N$th-order tensor}. Here  $I_{1}, I_{2}, \cdots, I_N$ are dimensions of the first, second, $\cdots$, $N$th way, respectively. The {\it order} of a tensor is the number of its dimensions. 
The scalars, vectors and matrices are respectively zeroth-order, first-order and second-order tensors.
The tensors of order three or higher are known as the higher-order tensors. These are denoted by calligraphic letters like  $\mc{A}$. An element of an $N$th order tensor $\mc{A}$ at $(i_{1},\cdots ,i_{N})$th position is denoted by $a_{i_{1}\cdots i_{N}}$. For more details, we refer to the recent books  \cite{ql, smilde, yw} on tensors.

There has been active research on tensors for the past four decades. But, a little research contributions on the theory and applications of generalized inverses of tensors are in the literature. In fact,  a generalized inverse called the {\it Moore--Penrose inverse of an even-order tensor} via the Einstein product was introduced first by  Sun { \it et al.} \cite{sun} in 2016. 
Then the authors find the minimum-norm least-squares solution of some multilinear systems by using the notion of Moore--Penrose inverse. In the next year, Behera and Mishra \cite{bm} continued the same study and proposed different types of generalized inverses of tensors. In 2018, Panigrahy and Mishra \cite{km} improved the definition of the Moore--Penrose inverse of an even-order tensor to a tensor of any order via the same product which also appeared in  \cite{lz} and \cite{scn}. Panigrahy and Mishra \cite{km} also introduced an extension of the Moore--Penrose inverse of a tensor called as the {\it Product Moore--Penrose inverse}.  The definition of the Moore--Penrose inverse of an arbitrary order tensor is recalled below.
\begin{definition}(Definition 3.3, \cite{lz} $\&$ Definition 1.1, \cite{km})\label{defmpi}\\
Let $\mc{X} \in {\C}^{I_{1}\times\cdots\times I_{M} \times J_{1}\times \cdots \times J_{N}}$. The tensor $\mc{Y} \in{\C}^{J_{1}\times\cdots\times J_{N} \times I_{1} \times\cdots\times I_{M}}$ satisfying the following four tensor equations:
\begin{eqnarray}
\mc{X}\n\mc{Y}\m\mc{X} &=& \mc{X};\label{mpeq1}\\
\mc{Y}\m\mc{X}\n\mc{Y} &=& \mc{Y};\label{mpeq2}\\
(\mc{X}\n\mc{Y})^{H} &=& \mc{X}\m\mc{Y};\label{mpeq3}\\
(\mc{Y}\m\mc{X})^{H} &=& \mc{Y}\n\mc{X},\label{mpeq4}
\end{eqnarray}
is defined as the \textbf{Moore--Penrose inverse} of $\mc{X}$, and is denoted by $\mc{X}^{\dg}$.
\end{definition}
\noindent In the above definition, $(\cdot)^{H}$ denotes the conjugate transpose of $(\cdot)$ and $\n$ denotes the {\it Einstein product} \cite{ein} of tensors, and is defined by 
\begin{equation*}\label{Eins}
(\mc{A}\n\mc{B})_{i_{1}\cdots i_{M}k_{1}\cdots k_{L}}
=\displaystyle\sum_{j_{1}\cdots j_{N}}a_{{i_{1}\cdots i_{M}}{j_{1}\cdots j_{N}}}b_{{j_{1}\cdots j_{N}}{k_{1}\cdots k_{L}}}
\end{equation*}
for tensors $\mc{A} \in \mathbb{C}^{I_{1}\times\cdots\times I_{M} \times J_{1}\times\cdots\times J_{N} }$ and $\mc{B} \in \mathbb{C}^{J_{1}\times\cdots\times J_{N} \times K_{1} \times\cdots\times K_{L} }$.
In the case of an even-order invertible tensor,  Definition \ref{defmpi} coincides with the notion of the inverse which was first introduced by Brazell {\it et al.} \cite{BraliNT13}. They also showed that such an inverse can be computed using the singular value decomposition of the same tensor (see Lemma 3.1, \cite{BraliNT13}). 
The idea of introducing generalized inverses of  tensors origins from the necessity of finding a solution of a given tensor multilinear system  (see \cite{ bm, jiw2, jiw1, sun}). 

The equality $(\mc{A}\n\mc{B})^\dg = \mc{B}^\dg \n
 \mc{A}^\dg $ for any two complex tensors $\mc{A}$ and $\mc{B}$ of arbitrary order,   is called as the {\it reverse-order law} for the Moore--Penrose inverse of tensors. Also called as two term reverse-order law. If the above equality contains the Einstein product of three different tensors, then it is called as the {\it triple  reverse-order law} (or three term reverse-order law). In the last section of \cite{bm}, the authors proposed two open problems. The first one is about the reverse-order law for the Moore--Penrose inverse of tensors and the second one is about full rank factorization of tensors. In 2018, Panigrahy {\it et al.} \cite{kbm} answered the first one for even-order tensors.  Again, Panigrahy and Mishra \cite{pm} added more results to the same theory but for arbitrary order tensors.  In the same year, Liang and Zheng \cite{lz}  solved the other problem.

 In 2017, Ji and Wei \cite{jiw1}  introduced another extension of the Moore--Penrose inverse of an even-order tensor called {\it weighted Moore--Penrose inverse}  and established the relation between the minimum-norm least-squares solution of a multilinear system and the weighted Moore--Penrose inverse. 
Very recently, Behera {\it et al.} \cite{bmm} proposed a few identities involving the weighted Moore--Penrose inverses of tensors. They also provided a method of computation of the weighted Moore--Penrose inverse using full rank decomposition of a tensor which was introduced by  Liang and Zheng \cite{lz}. Among other results, they attempted the problem of triple  reverse-order law mentioned in the conclusion section of \cite{km} for  the weighted Moore--Penrose inverse.
In this paper, 
 our aim is to study the {\it reverse-order law} for the weighted Moore--Penrose inverse of tensors via the Einstein product.

 In this context, the paper is organized as follows. Section 2 collects various useful definitions and results. A few properties of the weighted Moore--Penrose inverse of tensors are explained in Section 3. Section 4 contains all our main results and is devoted to the reverse-order law for the weighted Moore--Penrose inverse of tensors. It has two subsections, the first subsection contains a few necessary and sufficient conditions of reverse-order law for the square tensors while the second subsection is for arbitrary order tensors.
 
 \section{Prerequisites}\label{sec2}
Here, we collect some definitions and earlier results which will be used to prove the main results.
 We begin with the definition of a diagonal tensor. A tensor in ${\C}^{I_{1}\times \cdots \times I_{N}\times I_{1}\times \cdots \times I_{N}}$ with entries $(\mc{D})_{{i_1}...{i_N}{j_1}...{j_N}}$  is called a {\it diagonal
   tensor} if $d_{{i_1}...{i_N}{j_1}...{j_N}} = 0$ for $(i_1,\cdots,i_N) \neq (j_1,\cdots,j_N)$. A tensor $\mc{I}\in {\C}^{I_{1}\times\cdots\times I_{N}\times I_{1}\times\cdots\times I_{N}}$ with entries  $ (\mc{I})_{i_{1} \cdots i_{N}j_{1}\cdots j_{N}} = \prod_{k=1}^{N} \delta_{i_{k} j_{k}}$ is  called   an {\it identity tensor}  if
$ \delta_{i_{k}j_{k}}= \begin{cases}
 1, \text{ if } i_{k} = j_{k}\\
 0, \text{ otherwise}
 \end{cases}$.
 The {\it conjugate transpose} of a tensor $\mc{A}\in {\C}^{I_{1}\times\cdots\times I_{M}\times J_{1}\times \cdots \times J_{N}}$ is denoted by $\mc{A}^{H}$, and is defined as $(\mc{A}^{H})_{j_{1}\hdots j_{N}i_{1}\hdots i_{M}}=\overline{a}_{i_{1}\hdots i_{M}j_{1}\hdots j_{N}},$ 
where the over-line stands for the  conjugate of $a_{i_{1}\hdots i_{M}j_{1}\hdots j_{N}}$. A tensor $\mc{A}\in
\mathbb{C}^{I_1\times\cdots\times I_N \times I_1 \times\cdots\times
I_N}$ is {\it Hermitian}  if  $\mc{A}=\mc{A}^{H}$ and {\it skew-Hermitian} if $\mc{A}= - \mc{A}^{H}$. Further, a tensor
$\mc{A}\in \mathbb{C}^{I_1\times\cdots\times I_N \times I_1
\times\cdots\times I_N}$  is {\it unitary}  if  $\mc{A}\n
\mc{A}^{H}=\mc{A}^{H}\n \mc{A}=\mc{I}$, and {\it idempotent}  if $\mc{A}
\n \mc{A}= \mc{A}.$ Ji and Wei \cite{jiw1} define the class of Hermitian positive definite tensors as below.
\begin{definition}(Definition 1, \cite{jiw1})\label{dhpd}
Let $\mc{P}\in {\C}^{I_{1}\times \cdots \times I_{K}\times I_{1}\times \cdots \times I_{K}}$, if there exists a unitary tensor $\mc{U}$ such that 
\begin{equation}
    \mc{P}=\mc{U}\kp\mc{D}\kp\mc{U}^{H},
\end{equation}
where $\mc{D}$ is a diagonal tensor with positive diagonal entries, then $\mc{P}$ is said to be Hermitian positive definite.
\end{definition}
 For a Hermitian positive definite tensor $\mc{P}$ in the Definition \ref{dhpd}, let $\mc{D}^{1\slash2}$ be the diagonal tensor obtained from $\mc{D}$ by taking the square root of all its diagonal entries and define 
 \begin{equation*}
     \mc{P}^{1\slash2}=\mc{U}\kp\mc{D}^{1\slash2}\kp\mc{U}^{H},
 \end{equation*}
the square root of $\mc{P}$. Notice that $\mc{P}^{1\slash 2}$ is always non-singular and its inverse is denoted by $\mc{P}^{-1\slash2}$. Moreover, $ (\mc{P}^{1\slash2})^{H}=\mc{P}^{1\slash2}.$ In 2017, Ji and Wei \cite{jiw1} introduce the weighted Moore-Penrose inverse for even-order tensors, however, very recently Behera {\it et al.} \cite{bmm} defines it for any tensor and the definition is recalled next.
\begin{definition} (Definition 2.2, \cite{bmm})\label{wmpi}
Let $\mc{A}\in {\C}^{I_{1}\times \cdots\times I_{M}\times J_{1}\times \cdots\times J_{N}}$, and $\mc{M}\in {\C}^{I_{1}\times \cdots\times I_{M}\times I_{1}\times \cdots\times I_{M}}$, $\mc{N}\in {\C}^{J_{1}\times \cdots\times J_{N}\times J_{1}\times \cdots\times J_{N}}$ be Hermitian positive definite tensors. Then the unique tensor $\mc{X}$ in $\mc{C}^{J_{1}\times\cdots\times J_{N}\times I_{1}\times \cdots\times I_{M}}$ satisfying
\begin{eqnarray}
\mc{A}\n\mc{X}\m\mc{A}&=&\mc{A};\label{wmpie1}\\
\mc{X}\m\mc{A}\n\mc{X}&=&\mc{X};\label{wmpie2}\\
(\mc{M}\m\mc{A}\n\mc{X})^{H}&=&\mc{M}\m\mc{A}\n\mc{X};\label{wmpie3}\\
(\mc{N}\n\mc{X}\m\mc{A})^{H}&=&\mc{N}\n\mc{X}\m\mc{A},\label{wmpie4}
\end{eqnarray}
 is called the {\bf weighted Moore--Penrose} inverse of tensor $\mc{A}$ and this unique $\mc{X}$ is denoted by $\mc{A}^{\dg}_{\mc{M}\mc{N}}$.
\end{definition} 
In particular, when $\mc{M}$ and $\mc{N}$ are both identity tensors, then $\mc{A}^{\dg}_{\mc{I}_{\mc{M}}\mc{I}_{\mc{N}}}=\mc{A}^{\dg}$,
where $\mc{I}_{\mc{M}}\in {\C}^{I_{1}\times \cdots\times I_{M}\times I_{1}\times \cdots\times I_{M}}$ and $\mc{I}_{\mc{N}}\in {\C}^{J_{1}\times \cdots\times J_{N}\times J_{1}\times \cdots\times J_{N}}$ are identity tensors. The null space and the range space of $\mc{A}\in {\C}^{I_{1}\times \cdots\times I_{M}\times J_{1}\times \cdots\times J_{N}}$ are defined to be $N(\mc{A})=\{\mc{X}\in{\C}^{J_{1}\times \cdots \times J_{N}}:\mc{A}\n\mc{X}=\mc{O}\}$ and $R(\mc{A})=\{\mc{A}\n\mc{X}:\mc{X}\in{\C}^{J_{1}\times \cdots\times J_{N}}\}$, respectively.
Due to Ji and Wei, \cite{jiw1} we have the following property for weighted Moore--Penrose inverse of an even-order tensor.
\begin{lemma}(Lemma 2, \cite{jiw1})
Let $\mc{A}\in{\C}^{I_{1}\times \cdots\times I_{N}\times J_{1}\times \cdots\times J_{N}}$, and let $\mc{M}$ and $\mc{N}$ be Hermitian positive definite tensors in ${\C}^{I_{1}\times \cdots\times I_{N}\times I_{1}\times \cdots\times I_{N}}$ and
${\C}^{J_{1}\times \cdots\times J_{N}\times J_{1}\times \cdots\times J_{N}}$, respectively. Then
\begin{itemize}
    \item[(i)] $(\mc{A}^{\dg}_{\mc{M}\mc{N}})^{\dg}_{\mc{N}\mc{M}}=\mc{A}$,
    \item[(ii)] $(\mc{A}^{\dg}_{\mc{M}\mc{N}})^{H}=(\mc{A}^{H})^{\dg}_{\mc{N}^{-1}\mc{M}^{-1}}$.
\end{itemize}
\end{lemma}

\section{Weighted conjugate transpose}
Ji and Wei \cite{jiw1} defines the weighted conjugate transpose of a tensor whereas Behera {\it et al.} \cite{bmm} defines it formally for any tensor as below.
\begin{definition}(Definition 2.8, \cite{bmm})\label{wctd}
The weighted conjugate transpose of the tensor $\mc{A}\in {\C}^{I_{1}\times \cdots \times I_{M}\times J_{1}\times \cdots \times J_{N}}$, denoted by $\mc{A}^{\#}_{\mc{M}\mc{N}}$, is defined by
\begin{equation}
    \mc{A}^{\#}_{\mc{M}\mc{N}}=\mc{N}^{-1}\n\mc{A}^{H}\m\mc{M},
\end{equation}
where $\mc{M}\in {\C}^{I_{1}\times \cdots \times I_{M}\times I_{1}\times \cdots \times I_{M}} $ and $\mc{N}\in {\C}^{J_{1}\times \cdots \times J_{N}\times J_{1}\times \cdots \times J_{N}}$ are Hermitian positive definite tensors.
\end{definition}
Using Definition \ref{wctd}, one can easily verify the following properties of the weighted conjugate transpose.
\begin{lemma}
Let $\mc{A}, \mc{B}\in {\C}^{I_{1}\times \cdots \times I_{M}\times J_{1}\times \cdots \times J_{N}}$. Also, let $\mc{M}\in {\C}^{I_{1}\times \cdots\times I_{M}\times I_{1}\times \cdots\times I_{M}}$ and $\mc{N}\in {\C}^{J_{1}\times \cdots\times J_{N}\times J_{1}\times \cdots\times J_{N}}$ be Hermitian positive definite tensors. Then
\begin{eqnarray*}
(\mc{A}+\mc{B})^{\#}_{\mc{M}\mc{N}}&=&\mc{A}^{\#}_{\mc{M}\mc{N}}+\mc{B}^{\#}_{\mc{M}\mc{N}},\\
(\mc{A}^{\#}_{\mc{M}\mc{N}})^{H}&=&(\mc{A}^{H})^{\#}_{\mc{N}^{-1}\mc{M}^{-1}}.
\end{eqnarray*}
\end{lemma} 
Like the conjugate transpose of a tensor, the weighted conjugate transpose also satisfies reverse-order law and is stated below.
\begin{lemma}(Lemma 2.9, \cite{bmm})
Let $\mc{A}\in {\C}^{I_{1}\times \cdots \times I_{M}\times J_{1}\times \cdots \times J_{N}}$ and $\mc{B}\in {\C}^{J_{1}\times \cdots \times J_{N}\times K_{1}\times \cdots \times K_{L}}$. Also, let $\mc{M}\in {\C}^{I_{1}\times \cdots\times I_{M}\times I_{1}\times \cdots\times I_{M}}$, $\mc{N}\in {\C}^{J_{1}\times \cdots\times J_{N}\times J_{1}\times \cdots\times J_{N}}$ and $\mc{L}\in {\C}^{K_{1}\times \cdots\times K_{L}\times K_{1}\times \cdots\times K_{L}}$ be Hermitian positive definite tensors. Then
\begin{eqnarray*}
(\mc{A}^{\#}_{\mc{M}\mc{N}})^{\#}_{\mc{N}\mc{M}}&=&\mc{A},\\
    (\mc{A}\n\mc{B})^{\#}_{\mc{M}\mc{L}}&=&\mc{B}^{\#}_{\mc{N}\mc{L}}\n\mc{A}^{\#}_{\mc{M}\mc{N}}.
\end{eqnarray*}
\end{lemma}
The only tensor whose Einstein product with its weighted conjugate transpose results zero tensor is zero tensor. This is stated in next result. 
\begin{theorem}Let $\mc{A}\in {\C}^{I_{1}\times \cdots\times I_{M}\times J_{1}\times \cdots\times J_{N}}$, and $\mc{M}\in {\C}^{I_{1}\times \cdots\times I_{M}\times I_{1}\times \cdots\times I_{M}}$ and $\mc{N}\in {\C}^{J_{1}\times \cdots\times J_{N}\times J_{1}\times \cdots\times J_{N}}$ be Hermitian positive definite tensors. Then either of $\mc{A}^{\#}_{\mc{M}\mc{N}}\m\mc{A}=\mc{O}$ and $\mc{A}\n\mc{A}^{\#}_{\mc{M}\mc{N}}=\mc{O}$ implies $\mc{A}=\mc{O}$.
\end{theorem}
During the investigation of the reverse-order law for the weighted Moore--Penrose inverse of tensors, we observe that the weighted conjugate transpose of the Einstein product of a tensor with its weighted Moore--Penrose inverse remains unaltered.
\begin{theorem}
Let $\mc{A}\in {\C}^{I_{1}\times \cdots\times I_{M}\times J_{1}\times \cdots\times J_{N}}$, and $\mc{M}\in {\C}^{I_{1}\times \cdots\times I_{M}\times I_{1}\times \cdots\times I_{M}}$ and $\mc{N}\in {\C}^{J_{1}\times \cdots\times J_{N}\times J_{1}\times \cdots\times J_{N}}$ be Hermitian positive definite tensors. Then \begin{equation*}
    (\mc{A}\n\mc{A}^{\dg}_{\mc{M}\mc{N}})^{\#}_{\mc{M}\mc{M}}=\mc{A}\n\mc{A}^{\dg}_{\mc{M}\mc{N}}.
\end{equation*}
\end{theorem}
Similarly, we have the following result.
\begin{theorem}
Let $\mc{A}\in {\C}^{I_{1}\times \cdots\times I_{M}\times J_{1}\times \cdots\times J_{N}}$, and $\mc{M}\in {\C}^{I_{1}\times \cdots\times I_{M}\times I_{1}\times \cdots\times I_{M}}$ and $\mc{N}\in {\C}^{J_{1}\times \cdots\times J_{N}\times J_{1}\times \cdots\times J_{N}}$ be Hermitian positive definite tensors. Then 
\begin{equation*}
    (\mc{A}^{\dg}_{\mc{M}\mc{N}}\m \mc{A})^{\#}_{\mc{N}\mc{N}}=\mc{A}^{\dg}_{\mc{M}\mc{N}}\m \mc{A}.
\end{equation*}
\end{theorem}
Next, we provide the right  cancellation property of $\mc{A}^{\#}_{\mc{M}\mc{N}}$ for a tensor $\mc{A}\in{\C}^{I_{1}\times \cdots\times I_{M}\times J_{1}\times \cdots\times J_{N}}$.
\begin{lemma}\label{rct1}
Let $\mc{A} \in {\C}^{I_{1}\times\cdots\times I_{M} \times J_{1}\times\cdots\times J_{N} }$, $\mc{B}\in {\C}^{K_{1}\times\cdots\times K_{L} \times I_{1}\times\cdots\times I_{M} }$ and $\mc{C} \in {\C}^{K_{1}\times\cdots\times K_{L} \times I_{1}\times\cdots\times I_{M} }$. Also, let $\mc{M}\in {\C}^{I_{1}\times \cdots\times I_{M}\times I_{1}\times \cdots\times I_{M}}$ and $\mc{N}\in {\C}^{J_{1}\times \cdots\times J_{N}\times J_{1}\times \cdots\times J_{N}}$ be Hermitian positive definite tensors.
If $\mc{B}\m\mc{A} \n\mc{A}^{\#}_{\mc{M}\mc{N}} = \mc{C}\m\mc{A}\n\mc{A}^{\#}_{\mc{M}\mc{N}}$, then $\mc{B}\n \mc{A} = \mc{C}\n \mc{A}$.

\end{lemma}

Similarly, the left cancellation property is provided below.
\begin{lemma}\label{lct1}
Let $\mc{A} \in {\C}^{I_{1}\times\cdots\times I_{M} \times J_{1}\times\cdots\times J_{N} }$, $\mc{B}\in {\C}^{J_{1}\times\cdots\times J_{N} \times K_{1}\times\cdots\times K_{L} }$ and $\mc{C} \in {\C}^{J_{1}\times\cdots\times J_{N} \times K_{1}\times\cdots\times K_{L} }$. Also, let $\mc{M}\in {\C}^{I_{1}\times \cdots\times I_{M}\times I_{1}\times \cdots\times I_{M}}$ and $\mc{N}\in {\C}^{J_{1}\times \cdots\times J_{N}\times J_{1}\times \cdots\times J_{N}}$ be Hermitian positive definite tensors.
If $\mc{A}^{\#}_{\mc{M}\mc{N}}\m\mc{A}\n\mc{B} = \mc{A}^{\#}_{\mc{M}\mc{N}}\m\mc{A}\n\mc{C}$, then $\mc{A}\n \mc{B} = \mc{A}\n \mc{C}$.
\end{lemma}
A  sufficient condition for the commutativity of $\mc{A}^{\dg}_{\mc{M}\mc{N}}\m\mc{A}$ and $\mc{B}\lp\mc{B}^{\#}_{\mc{M}\mc{N}}$ is presented below.
\begin{lemma}\label{lm3.13}
Let $\mc{A} \in \mathbb{C}^{I_{1}\times\cdots\times I_{M} \times J_{1}\times\cdots\times J_{N}}$ and $\mc{B}\in \mathbb{C}^{J_{1}\times\cdots\times J_{N} \times K_{1}\times\cdots\times K_{L}}$. Also, let $\mc{M}\in {\C}^{I_{1}\times \cdots\times I_{M}\times I_{1}\times \cdots\times I_{M}}$, $\mc{N}\in {\C}^{J_{1}\times \cdots\times J_{N}\times J_{1}\times \cdots\times J_{N}}$ and $\mc{L}\in {\C}^{K_{1}\times \cdots\times K_{L}\times K_{1}\times \cdots\times K_{L}}$ be Hermitian positive definite tensors. If $\mc{A}^{\dg}_{\mc{M}\mc{N}} \m \mc{A} \n \mc{B} \lp \mc{B}^{\#}_{\mc{N}\mc{L}} \n \mc{A}^{\#}_{\mc{M}\mc{N}}=\mc{B}\lp \mc{B}^{\#}_{\mc{N}\mc{L}} \n\mc{A}^{\#}_{\mc{M}\mc{N}}$, then $\mc{A}^{\dg}_{\mc{M}\mc{N}}\m\mc{A}$ commutes with $\mc{B}\lp\mc{B}^{\#}_{\mc{N}\mc{L}}$.
\end{lemma}

Similarly, we can have the following result.
\begin{lemma}
Let $\mc{A} \in \mathbb{C}^{I_{1}\times\cdots\times I_{M} \times J_{1}\times\cdots\times J_{N} }$ and $\mc{B}\in \mathbb{C}^{J_{1}\times\cdots\times J_{N} \times K_{1}\times\cdots\times K_{L} }$. Also, let $\mc{M}\in {\C}^{I_{1}\times \cdots\times I_{M}\times I_{1}\times \cdots\times I_{M}}$, $\mc{N}\in {\C}^{J_{1}\times \cdots\times J_{N}\times J_{1}\times \cdots\times J_{N}}$ and $\mc{L}\in {\C}^{K_{1}\times \cdots\times K_{L}\times K_{1}\times \cdots\times K_{L}}$ be Hermitian positive definite tensors. If $\mc{B}\lp\mc{B}^{\dg}_{\mc{N}\mc{L}}\n\mc{A}^{\#}_{\mc{M}\mc{N}}\m\mc{A}\n\mc{B}=\mc{A}^{\#}_{\mc{M}\mc{N}}\m\mc{A}\n\mc{B}$, then $\mc{B}\lp\mc{B}^{\dg}_{\mc{N}\mc{L}}$ commutes with $\mc{A}^{\#}_{\mc{M}\mc{N}}\m\mc{A}$.
\end{lemma}

Equivalent conditions for the commutativity of $\mc{A}^{\dg}_{\mc{M}\mc{N}}\m\mc{A}$ and $\mc{B}\lp\mc{B}^{\dg}_{\mc{N}\mc{L}}$ are obtained next.
\begin{lemma}
Let $\mc{A} \in \mathbb{C}^{I_{1}\times\cdots\times I_{M} \times J_{1}\times\cdots\times J_{N} }$ and $\mc{B}\in \mathbb{C}^{J_{1}\times\cdots\times J_{N} \times K_{1}\times\cdots\times K_{L} }$. Also, let $\mc{M}\in {\C}^{I_{1}\times \cdots\times I_{M}\times I_{1}\times \cdots\times I_{M}}$, $\mc{N}\in {\C}^{J_{1}\times \cdots\times J_{N}\times J_{1}\times \cdots\times J_{N}}$ and $\mc{L}\in {\C}^{K_{1}\times \cdots\times K_{L}\times K_{1}\times \cdots\times K_{L}}$ be Hermitian positive definite tensors. Then the commutativity of $\mc{A}^{\dg}_{\mc{M}\mc{N}}\m\mc{A}$ and $\mc{B}\lp\mc{B}^{\dg}_{\mc{N}\mc{L}}$ is equivalent to either of the conditions
\begin{equation}\label{eq9}
\mc{A}^{\dg}_{\mc{M}\mc{N}} \m\mc{A} \n\mc{B} \lp\mc{B}^{\dg}_{\mc{N}\mc{L}}\n\mc{A}^{\#}_{\mc{M}\mc{N}}=\mc{B}\lp\mc{B}^{\dg}_{\mc{N}\mc{L}}\n \mc{A}^{\#}_{\mc{M}\mc{N}}
\end{equation}
and
\begin{equation}\label{eq10}
\mc{B} \lp\mc{B}^{\dg}_{\mc{N}\mc{L}} \n\mc{A}^{\dg}_{\mc{M}\mc{N}} \m\mc{A} \n\mc{B} =\mc{A}^{\dg}_{\mc{M}\mc{N}}\m\mc{A}\n \mc{B}.
\end{equation}

\end{lemma}

The next result provides an absorbing property of a tensor which coincides with its weighted conjugate transpose.
\begin{lemma}\label{ir}
Let  $\mc{P}\in {\C}^{I_{1}\times\cdots\times I_{N}\times I_{1}\times \cdots \times I_{N}}$ such that $\mc{N}\n\mc{P}=(\mc{N}\n\mc{P})^{H}$. Also, let $\mc{M}$ and $\mc{N}$ be Hermitian positive definite tensors of appropriate size. If $\mc{P}\n\mc{Q}=\mc{Q}$ for $\mc{Q}\in {\C}^{I_{1}\times\cdots\times I_{N}\times J_{1}\times \cdots \times J_{M}}$, then 
\begin{equation*}\label{ir1}
   \mc{Q}^{\dg}_{\mc{M}\mc{N}}\n\mc{P}=\mc{Q}^{\dg}_{\mc{M}\mc{N}}. 
\end{equation*}
\end{lemma}
\noindent In a similar way, we have the following result.
\begin{lemma}\label{irt2}
Let  $\mc{P}\in {\C}^{I_{1}\times\cdots\times I_{N}\times I_{1}\times \cdots \times I_{N}}$ such that $\mc{N}\n\mc{P}=(\mc{N}\n\mc{P})^{H}$. Also, let $\mc{M}$ and $\mc{N}$ be Hermitian positive definite tensors of appropriate size. If $\mc{Q}\n\mc{P}=\mc{Q}$ for $\mc{Q}\in {\C}^{J_{1}\times \cdots \times J_{M}\times I_{1}\times\cdots\times I_{N}}$, then
\begin{equation*}\label{ir2}
   \mc{P}\n\mc{Q}^{\dg}_{\mc{M}\mc{N}}=\mc{Q}^{\dg}_{\mc{M}\mc{N}}.
\end{equation*}
\end{lemma}

\section{Main results}
\subsection{Reverse-order law for square tensors}
In this subsection, we provide some sufficient conditions of the reverse-order law for the weighted Moore--Penrose inverse of the Einstein product of two square tensors.
\begin{theorem}\label{rolsct1}
Let $\mc{A}$, $\mc{B} \in {\C}^{I_{1}\times \cdots \times I_{N}\times I_{1}\times \cdots \times I_{N}}$. Also, let $\mc{M},~\mc{N}\in{\C}^{I_{1}\times \cdots \times I_{N}\times I_{1}\times \cdots \times I_{N}}$ be Hermitian positive definite tensors. If 
\begin{eqnarray}
\mc{A}\n(\mc{B}\n\mc{B}^{\dg}_{\mc{M}\mc{N}})&=&(\mc{B}\n\mc{B}^{\dg}_{\mc{M}\mc{N}})\n\mc{A},\label{rolsc1}\\
\mc{A}^{\dg}_{\mc{M}\mc{N}}\n(\mc{B}\n\mc{B}^{\dg}_{\mc{M}\mc{N}})&=&(\mc{B}\n\mc{B}^{\dg}_{\mc{M}\mc{N}})\n\mc{A}^{\dg}_{\mc{M}\mc{N}},\label{rolsc2}\\
\mc{B}\n(\mc{A}^{\dg}_{\mc{M}\mc{N}}\n\mc{A})&=&(\mc{A}^{\dg}_{\mc{M}\mc{N}}\n\mc{A})\n\mc{B},\label{rolsc3}\\
\mc{B}^{\dg}_{\mc{M}\mc{N}}\n(\mc{A}^{\dg}_{\mc{M}\mc{N}}\n\mc{A})&=&(\mc{A}^{\dg}_{\mc{M}\mc{N}}\n\mc{A})\n\mc{B}^{\dg}_{\mc{M}\mc{N}},\label{rolsc4}
\end{eqnarray}
then $(\mc{A}\n\mc{B})^{\dg}_{\mc{M}\mc{N}}=\mc{B}^{\dg}_{\mc{M}\mc{N}}\n\mc{A}^{\dg}_{\mc{M}\mc{N}}$.
\end{theorem}
\begin{proof}
Suppose that Equations \eqref{rolsc1}-\eqref{rolsc4} hold. Let $\mc{X}=\mc{A}\n\mc{B}$ and $\mc{Y}=\mc{B}^{\dg}_{\mc{M}\mc{N}}\n\mc{A}^{\dg}_{\mc{M}\mc{N}}$, then with the help of Equations \eqref{rolsc3} and \eqref{rolsc4}, we have $\mc{X}\n\mc{Y}\n\mc{X}=\mc{X}$ and using Equations \eqref{rolsc1} and \eqref{rolsc2}, we have $\mc{Y}\n\mc{X}\n\mc{Y}=\mc{Y}.$ Using Equation \eqref{rolsc1}, we have $(\mc{M}\n\mc{X}\n\mc{Y})^{H}=[(\mc{M}\n\mc{B}\n\mc{B}^{\dg}_{\mc{M}\mc{N}})\n\mc{M}^{-1}\n(\mc{M}\n\mc{A}\n\mc{A}^{\dg}_{\mc{M}\mc{N}})]^{H}$, which results $(\mc{M}\n\mc{X}\n\mc{Y})^{H}=\mc{M}\n\mc{X}\n\mc{Y}$ due to Equation \eqref{rolsc2}. Using Equation \eqref{rolsc4} we have $(\mc{N}\n\mc{Y}\n\mc{X})^{H}=[\mc{N}\n(\mc{A}^{\dg}_{\mc{M}\mc{N}}\n\mc{A})*_{N}\mc{N}^{-1} \n(\mc{N}\n\mc{B}^{\dg}_{\mc{M}\mc{N}}\n\mc{B})]^{H}$, which results $(\mc{N}\n\mc{Y}\n\mc{X})^{H}=\mc{N}\n\mc{Y}\n\mc{X}$ due to Equation \eqref{rolsc3}. 
Therefore, by Definition \ref{wmpi}, we get $\mc{X}^{\dg}_{\mc{M}\mc{N}}=\mc{Y}$, i.e., $(\mc{A}\n\mc{B})^{\dg}_{\mc{M}\mc{N}}=\mc{B}^{\dg}_{\mc{M}\mc{N}}\n\mc{A}^{\dg}_{\mc{M}\mc{N}}$.
\end{proof}
We replace the conditions \eqref{rolsc1} and \eqref{rolsc2} of Theorem \ref{rolsct1} by a single condition, and is presented next.
\begin{theorem}
Let $\mc{A}$, $\mc{B} \in {\C}^{I_{1}\times \cdots \times I_{N}\times I_{1}\times \cdots \times I_{N}}$. Also, let $\mc{M},~\mc{N}\in{\C}^{I_{1}\times \cdots \times I_{N}\times I_{1}\times \cdots \times I_{N}}$ be Hermitian positive definite tensors. If 
\begin{itemize}
    \item[(i)] $(\mc{M}\n\mc{A}\n\mc{B}\n\mc{B}^{\dg}_{\mc{M}\mc{N}}\n\mc{A}^{\dg}_{\mc{M}\mc{N}})^{H}=\mc{M}\n\mc{A}\n\mc{B}\n\mc{B}^{\dg}_{\mc{M}\mc{N}}\n\mc{A}^{\dg}_{\mc{M}\mc{N}}$,
    \item[(ii)] $\mc{B}\n(\mc{A}^{\dg}_{\mc{M}\mc{N}}\n\mc{A})=(\mc{A}^{\dg}_{\mc{M}\mc{N}}\n\mc{A})\n\mc{B}$,
    \item[(iii)] $\mc{B}^{\dg}_{\mc{M}\mc{N}}\n(\mc{A}^{\dg}_{\mc{M}\mc{N}}\n\mc{A})=(\mc{A}^{\dg}_{\mc{M}\mc{N}}\n\mc{A})\n\mc{B}^{\dg}_{\mc{M}\mc{N}}$,
\end{itemize}
then $(\mc{A}\n\mc{B})^{\dg}_{\mc{M}\mc{N}}=\mc{B}^{\dg}_{\mc{M}\mc{N}}\n\mc{A}^{\dg}_{\mc{M}\mc{N}}$.
\end{theorem}
Similarly, one can replace the conditions \eqref{rolsc3} and \eqref{rolsc4} of Theorem \ref{rolsct1} by a single condition as follows.
\begin{theorem}
Let $\mc{A}$, $\mc{B} \in {\C}^{I_{1}\times \cdots \times I_{N}\times I_{1}\times \cdots \times I_{N}}$. Also, let $\mc{M},~\mc{N}\in{\C}^{I_{1}\times \cdots \times I_{N}\times I_{1}\times \cdots \times I_{N}}$ be Hermitian positive definite tensors. If 
\begin{itemize}
    \item[(i)] $\mc{A}\n(\mc{B}\n\mc{B}^{\dg}_{\mc{M}\mc{N}})=(\mc{B}\n\mc{B}^{\dg}_{\mc{M}\mc{N}})\n\mc{A}$,
    \item[(ii)] $\mc{A}^{\dg}_{\mc{M}\mc{N}}\n(\mc{B}\n\mc{B}^{\dg}_{\mc{M}\mc{N}})=(\mc{B}\n\mc{B}^{\dg}_{\mc{M}\mc{N}})\n\mc{A}^{\dg}_{\mc{M}\mc{N}}$,
    \item[(iii)] $(\mc{N}\n\mc{B}^{\dg}_{\mc{M}\mc{N}}\n\mc{A}^{\dg}_{\mc{M}\mc{N}}\n\mc{A}\n\mc{B})^{H}=\mc{N}\n\mc{B}^{\dg}_{\mc{M}\mc{N}}\n\mc{A}^{\dg}_{\mc{M}\mc{N}}\n\mc{A}\n\mc{B}$,
\end{itemize}
then $(\mc{A}\n\mc{B})^{\dg}_{\mc{M}\mc{N}}=\mc{B}^{\dg}_{\mc{M}\mc{N}}\n\mc{A}^{\dg}_{\mc{M}\mc{N}}$.
\end{theorem}
The next result confirms that whenever the tensor $\mc{A}$ satisfies $\mc{M}\n\mc{A}^{\dg}_{\mc{M}\mc{N}}=(\mc{M}\n\mc{A})^{H}$, then \eqref{rolsc3} and \eqref{rolsc4} are sufficient to hold the reverse-order law.
\begin{theorem}
Let $\mc{A}$, $\mc{B} \in {\C}^{I_{1}\times \cdots \times I_{N}\times I_{1}\times \cdots \times I_{N}}$. Also, let $\mc{M},~\mc{N}\in{\C}^{I_{1}\times \cdots \times I_{N}\times I_{1}\times \cdots \times I_{N}}$ be Hermitian positive definite tensors and $\mc{M}\n\mc{A}^{\dg}_{\mc{M}\mc{N}}=(\mc{M}\n\mc{A})^{H}$. If 
\begin{eqnarray}
\mc{B}\n(\mc{A}^{\dg}_{\mc{M}\mc{N}}\n\mc{A})&=&(\mc{A}^{\dg}_{\mc{M}\mc{N}}\n\mc{A})\n\mc{B},\label{rolah1}\\
\mc{B}^{\dg}_{\mc{M}\mc{N}}\n(\mc{A}^{\dg}_{\mc{M}\mc{N}}\n\mc{A})&=&(\mc{A}^{\dg}_{\mc{M}\mc{N}}\n\mc{A})\n\mc{B}^{\dg}_{\mc{M}\mc{N}},\label{rolah2}
\end{eqnarray}
then $(\mc{A}\n\mc{B})^{\dg}_{\mc{M}\mc{N}}=\mc{B}^{\dg}_{\mc{M}\mc{N}}\n\mc{A}^{\dg}_{\mc{M}\mc{N}}$.
\end{theorem}
 Similarly, conditions \eqref{rolsc1} and \eqref{rolsc2} are sufficient to hold the reverse-order law whenever $\mc{N}\n\mc{B}^{\dg}_{\mc{M}\mc{N}}=(\mc{N}\n\mc{B})^{H}$.
\begin{theorem}
Let $\mc{A}$, $\mc{B} \in {\C}^{I_{1}\times \cdots \times I_{N}\times I_{1}\times \cdots \times I_{N}}$. Also, let $\mc{M},~\mc{N}\in{\C}^{I_{1}\times \cdots \times I_{N}\times I_{1}\times \cdots \times I_{N}}$ be Hermitian positive definite tensors and $\mc{N}\n\mc{B}^{\dg}_{\mc{M}\mc{N}}=(\mc{N}\n\mc{B})^{H}$. If 
\begin{eqnarray*}
\mc{A}\n(\mc{B}\n\mc{B}^{\dg}_{\mc{M}\mc{N}})&=&(\mc{B}\n\mc{B}^{\dg}_{\mc{M}\mc{N}})\n\mc{A},\\
\mc{A}^{\dg}_{\mc{M}\mc{N}}\n(\mc{B}\n\mc{B}^{\dg}_{\mc{M}\mc{N}})&=&(\mc{B}\n\mc{B}^{\dg}_{\mc{M}\mc{N}})\n\mc{A}^{\dg}_{\mc{M}\mc{N}},
\end{eqnarray*}
then $(\mc{A}\n\mc{B})^{\dg}_{\mc{M}\mc{N}}=\mc{B}^{\dg}_{\mc{M}\mc{N}}\n\mc{A}^{\dg}_{\mc{M}\mc{N}}$.
\end{theorem}
Next, we provide an interesting result which confirms that only one condition among \eqref{rolsc1}-\eqref{rolsc4} is sufficient to hold the reverse-order law for the tensors $\mc{A}$ and $\mc{B}$ satisfying $\mc{M}\n\mc{A}^{\dg}_{\mc{M}\mc{N}}=(\mc{M}\n\mc{A})^{H}$ and $\mc{N}\n\mc{B}^{\dg}_{\mc{M}\mc{N}}=(\mc{N}\n\mc{B})^{H}$.
\begin{theorem}
Let $\mc{A}$, $\mc{B} \in {\C}^{I_{1}\times \cdots \times I_{N}\times I_{1}\times \cdots \times I_{N}}$. Also, let $\mc{M},~\mc{N}\in{\C}^{I_{1}\times \cdots \times I_{N}\times I_{1}\times \cdots \times I_{N}}$ be Hermitian positive definite tensors and $\mc{M}\n\mc{A}^{\dg}_{\mc{M}\mc{N}}=(\mc{M}\n\mc{A})^{H}$ and $\mc{N}\n\mc{B}^{\dg}_{\mc{M}\mc{N}}=(\mc{N}\n\mc{B})^{H}$. If at least one of the following holds
\begin{eqnarray}
\mc{A}\n(\mc{B}\n\mc{B}^{\dg}_{\mc{M}\mc{N}})&=&(\mc{B}\n\mc{B}^{\dg}_{\mc{M}\mc{N}})\n\mc{A},\\
\mc{A}^{\dg}_{\mc{M}\mc{N}}\n(\mc{B}\n\mc{B}^{\dg}_{\mc{M}\mc{N}})&=&(\mc{B}\n\mc{B}^{\dg}_{\mc{M}\mc{N}})\n\mc{A}^{\dg}_{\mc{M}\mc{N}},\\
\mc{B}\n(\mc{A}^{\dg}_{\mc{M}\mc{N}}\n\mc{A})&=&(\mc{A}^{\dg}_{\mc{M}\mc{N}}\n\mc{A})\n\mc{B},\\
\mc{B}^{\dg}_{\mc{M}\mc{N}}\n(\mc{A}^{\dg}_{\mc{M}\mc{N}}\n\mc{A})&=&(\mc{A}^{\dg}_{\mc{M}\mc{N}}\n\mc{A})\n\mc{B}^{\dg}_{\mc{M}\mc{N}},
\end{eqnarray}
then $(\mc{A}\n\mc{B})^{\dg}_{\mc{M}\mc{N}}=\mc{B}^{\dg}_{\mc{M}\mc{N}}\n\mc{A}^{\dg}_{\mc{M}\mc{N}}$.
\end{theorem}

\subsection{Reverse-order law for arbitrary tensors}
In this subsection, we provide some necessary and sufficient conditions of the reverse-order law for the weighted Moore--Penrose inverse of the Einstein product of two tensors. 
The very first result is proved in \cite{bmm} using the range and null space of a tensor, however we present a new proof without using the notion of the range and null space of a tensor.
\begin{theorem}\label{wrol1}
Let $\mc{A}\in {\C}^{I_{1}\times \cdots\times I_{M}\times J_{1}\times \cdots \times J_{N}}$ and  $\mc{B}\in {\C}^{J_{1}\times \cdots \times J_{N}\times K_{1}\times \cdots\times K_{L}}$. Also, let $\mc{M}\in\mc{C}^{I_{1}\times \cdots\times I_{M}\times I_{1}\times \cdots\times I_{M}}$, $\mc{N}\in \mc{C}^{J_{1}\times \cdots\times J_{N}\times J_{1}\times \cdots\times J_{N}}$ and  $\mc{L}\in \mc{C}^{K_{1}\times \cdots\times K_{L}\times K_{1}\times \cdots\times K_{L}}$ be Hermitian positive definite tensors. Then 
$$(\mc{A}\n\mc{B})^{\dg}_{\mc{M}\mc{L}}=\mc{B}^{\dg}_{\mc{N}\mc{L}}\n\mc{A}^{\dg}_{\mc{M}\mc{N}}$$
if and only if 
\begin{eqnarray}
\mc{A}^{\dg}_{\mc{M}\mc{N}}\m\mc{A}\n\mc{B}\lp\mc{B}^{\#}_{\mc{N}\mc{L}}\n\mc{A}^{\#}_{\mc{M}\mc{N}}&=&\mc{B}\lp\mc{B}^{\#}_{\mc{N}\mc{L}}\n\mc{A}^{\#}_{\mc{M}\mc{N}},\label{role1}\\
\mc{B}\lp\mc{B}^{\dg}_{\mc{N}\mc{L}}\n\mc{A}^{\#}_{\mc{M}\mc{N}}\m\mc{A}\n\mc{B}&=&\mc{A}^{\#}_{\mc{M}\mc{N}}\m\mc{A}\n\mc{B}.\label{role2}
\end{eqnarray}
\end{theorem}
\begin{proof}
Suppose that $(\mc{A}\n\mc{B})^{\dg}_{\mc{M}\mc{L}}=\mc{B}^{\dg}_{\mc{N}\mc{L}}\n\mc{A}^{\dg}_{\mc{M}\mc{N}}$. We have
\begin{equation*}
    (\mc{A}\n\mc{B})^{\#}_{\mc{M}\mc{L}}=(\mc{A}\n\mc{B})^{\dg}_{\mc{M}\mc{L}}\m(\mc{A}\n\mc{B})\lp(\mc{A}\n\mc{B})^{\#}_{\mc{M}\mc{L}},
\end{equation*}
which results
\begin{eqnarray*}
\mc{B}^{\#}_{\mc{N}\mc{L}} \n\mc{A}^{\#}_{\mc{M}\mc{N}}&=&\mc{B}^{\dg}_{\mc{N}\mc{L}}\n\mc{A}^{\dg}_{\mc{M}\mc{N}}\m\mc{A}\n\mc{B}\lp\mc{B}^{\#}_{\mc{N}\mc{L}}\n\mc{A}^{\#}_{\mc{N}\mc{L}}.
\end{eqnarray*}
Pre-multiply both sides with $\mc{A}\n\mc{B}\lp\mc{B}^{\#}_{\mc{N}\mc{L}}\n\mc{B}$, we get
\begin{eqnarray*}
\mc{A}\n\mc{B}\lp\mc{B}^{\#}_{\mc{N}\mc{L}}\n\mc{B}\lp\mc{B}^{\#}_{\mc{N}\mc{L}} \n\mc{A}^{\#}_{\mc{M}\mc{N}}
&=&\mc{A}\n\mc{B}\lp\mc{B}^{\#}_{\mc{N}\mc{L}}\n\mc{A}^{\dg}_{\mc{M}\mc{N}}\m\mc{A}\n\mc{B}\lp\mc{B}^{\#}_{\mc{N}\mc{L}}\n\mc{A}^{\#}_{\mc{M}\mc{N}}.
\end{eqnarray*}
From which, we get
\begin{eqnarray*}
\mc{A}\n\mc{B}\lp\mc{B}^{\#}_{\mc{N}\mc{L}}\n(\mc{I}-\mc{A}^{\dg}_{\mc{M}\mc{N}}\m\mc{A})\n\mc{B}\lp\mc{B}^{\#}_{\mc{N}\mc{L}}\n\mc{A}^{\#}_{\mc{M}\mc{N}}&=&\mc{O}.
\end{eqnarray*}
Since $\mc{I}-
\mc{A}^{\dg}_{\mc{M}\mc{N}}\m\mc{A}$ is idempotent and $(\mc{I}-
\mc{A}^{\dg}_{\mc{M}\mc{N}}\m\mc{A})^{\#}_{\mc{N}\mc{N}}=\mc{I}-
\mc{A}^{\dg}_{\mc{M}\mc{N}}\m\mc{A}$. So, the last equation can be written as
\begin{equation*}
     [(\mc{I}-
\mc{A}^{\dg}_{\mc{M}\mc{N}}\m\mc{A})\n\mc{B}\lp\mc{B}^{\#}_{\mc{N}\mc{L}}\n \mc{A}^{\#}_{\mc{M}\mc{N}}]^{\#}_{\mc{N}\mc{M}}\n(\mc{I}-
\mc{A}^{\dg}_{\mc{M}\mc{N}}\m\mc{A})\n\mc{B}\lp\mc{B}^{\#}_{\mc{N}\mc{L}}\n\mc{A}^{\#}_{\mc{M}\mc{N}}=\mc{O}.
\end{equation*}
Thus,
\begin{equation*}
    \mc{A}^{\dg}_{\mc{M}\mc{N}}\m\mc{A}\n\mc{B}\lp\mc{B}^{\#}_{\mc{N}\mc{L}}\n\mc{A}^{\#}_{\mc{M}\mc{N}}=\mc{B}\lp\mc{B}^{\#}_{\mc{N}\mc{L}}\n\mc{A}^{\#}_{\mc{M}\mc{N}}.
\end{equation*}
Again,
\begin{eqnarray*}
\mc{A}\n\mc{B} 
&=& (\mc{A}^{\dg}_{\mc{M}\mc{N}})^{\#}_{\mc{N}\mc{M}} \n(\mc{B}^{\dg}_{\mc{N}\mc{L}})^{\#}_{\mc{L}\mc{N}} \lp\mc{B}^{\#}_{\mc{N}\mc{L}} \n\mc{A}^{\#}_{\mc{M}\mc{N}} \m\mc{A}\n\mc{B}\\
&=& (\mc{A}^{\dg}_{\mc{M}\mc{N}})^{\#}_{\mc{N}\mc{M}} \n\mc{B}\lp\mc{B}^{\dg}_{\mc{N}\mc{L}} \n\mc{A}^{\#}_{\mc{M}\mc{N}} \m\mc{A}\n\mc{B}.
\end{eqnarray*}
Now, pre-multiplying $\mc{B}^{\#}_{\mc{N}\mc{L}}\n\mc{A}^{\#}_{\mc{M}\mc{N}} \m\mc{A}\n\mc{A}^{\#}_{\mc{M}\mc{N}}$ to $\mc{A}\n\mc{B}=(\mc{A}^{\dg}_{\mc{M}\mc{N}})^{\#}_{\mc{N}\mc{M}} \n\mc{B}\lp\mc{B}^{\dg}_{\mc{N}\mc{L}} \n\mc{A}^{\#}_{\mc{M}\mc{N}}\m\mc{A}\n\mc{B}$, we get
\begin{eqnarray*}
\mc{B}^{\#}_{\mc{N}\mc{L}}\n\mc{A}^{\#}_{\mc{M}\mc{N}} \m\mc{A}\n\mc{A}^{\#}_{\mc{M}\mc{N}}\m\mc{A}\n\mc{B} &=&\mc{B}^{\#}_{\mc{N}\mc{L}}\n\mc{A}^{\#}_{\mc{M}\mc{N}} \m \mc{A} \n\mc{B}\lp\mc{B}^{\dg}_{\mc{N}\mc{L}} \n\mc{A}^{\#}_{\mc{M}\mc{N}} \m\mc{A}\n\mc{B},
\end{eqnarray*}
which gives
\begin{eqnarray*}
\mc{B}^{\#}_{\mc{N}\mc{L}}\n\mc{A}^{\#}_{\mc{M}\mc{N}} \m\mc{A}\n(\mc{I}-\mc{B}\lp\mc{B}^{\dg}_{\mc{N}\mc{L}})\n\mc{A}^{\#}_{\mc{M}\mc{N}}\m\mc{A}\n\mc{B}&=&\mc{O}.
\end{eqnarray*}
Since $\mc{I}-\mc{B}\lp\mc{B}^{\dg}_{\mc{N}\mc{L}}$ is idempotent and $(\mc{I} -\mc{B} \lp\mc{B}^{\dg}_{\mc{N}\mc{L}})^{\#}_{\mc{N}\mc{N}}=\mc{I} -\mc{B} \lp\mc{B}^{\dg}_{\mc{N}\mc{L}}$. So
\begin{eqnarray*}
   [(\mc{I}-\mc{B}\lp\mc{B}^{\dg}_{\mc{N}\mc{L}}) \n\mc{A}^{\#}_{\mc{M}\mc{N}}\m\mc{A}\n\mc{B}]^{\#}_{\mc{N}\mc{L}}\n(\mc{I}-\mc{B}\lp\mc{B}^{\dg}_{\mc{N}\mc{L}}) \n\mc{A}^{\#}_{\mc{M}\mc{N}}\m\mc{A}\n\mc{B} &=& \mc{O}
\end{eqnarray*}
Thus,
\begin{equation*}
    \mc{B}\lp\mc{B}^{\dg}_{\mc{N}\mc{L}}\n\mc{A}^{\#}_{\mc{M}\mc{N}}\m\mc{A}\n\mc{B}=\mc{A}^{\#}_{\mc{M}\mc{N}}\m\mc{A}\n\mc{B}.
\end{equation*}
Conversely, suppose that Equations \eqref{role1} and \eqref{role2} hold, i.e.,
\begin{eqnarray*}
\mc{A}^{\dg}_{\mc{M}\mc{N}}\m\mc{A}\n\mc{B}\lp\mc{B}^{\#}_{\mc{N}\mc{L}}\n\mc{A}^{\#}_{\mc{M}\mc{N}}&=&\mc{B}\lp\mc{B}^{\#}_{\mc{N}\mc{L}}\n\mc{A}^{\#}_{\mc{M}\mc{N}},\\
\mc{B}\lp\mc{B}^{\dg}_{\mc{N}\mc{L}}\n\mc{A}^{\#}_{\mc{M}\mc{N}}\m\mc{A}\n\mc{B}&=&\mc{A}^{\#}_{\mc{M}\mc{N}}\m\mc{A}\n\mc{B}.
\end{eqnarray*}
Pre-multiplying and post-multiplying Equation \eqref{role1} by $\mc{B}^{\dg}_{\mc{N}\mc{L}}$ and $((\mc{A}\n\mc{B})^{\#}_{\mc{M}\mc{L}})^{\dg}_{\mc{L}\mc{M}}$, respectively, we get
\begin{equation}
\mc{B}^{\dg}_{\mc{N}\mc{L}}\n\mc{A}^{\dg}_{\mc{M}\mc{N}}\m\mc{A}\n\mc{B} =  (\mc{A}\n\mc{B})^{\dg}_{\mc{M}\mc{L}}\m (\mc{A} \n \mc{B}).\label{eq21}
\end{equation}
Taking the weighted conjugate transpose  of Equation \eqref{role2}, we have 
\begin{equation}\label{eq22}
\mc{B}^{\#}_{\mc{N}\mc{L}} \n\mc{A}^{\#}_{\mc{M}\mc{N}} \m\mc{A} \n\mc{B} \lp\mc{B}^{\dg}_{\mc{N}\mc{L}}  =  \mc{B}^{\#}_{\mc{N}\mc{L}} \n\mc{A}^{\#}_{\mc{M}\mc{N}} \m\mc{A}.
\end{equation}
Pre-multiplying and post-multiplying Equation \eqref{eq22} by $((\mc{A}\n\mc{B})^{\#}_{\mc{M}\mc{L}})^{\dg}_{\mc{L}\mc{M}}$ and $\mc{A}^{\dg}_{\mc{M}\mc{N}}$, respectively, we obtain
\begin{equation}
\mc{A} \n\mc{B} \lp\mc{B}^{\dg}_{\mc{N}\mc{L}}\n\mc{A}^{\dg}_{\mc{M}\mc{N}} =  \mc{A} \n\mc{B}\lp(\mc{A}\n\mc{B})^{\dg}_{\mc{M}\mc{L}}.\label{eq23}
\end{equation}

In order to show
$(\mc{A}\n\mc{B})^{\dg}_{\mc{M}\mc{L}} = \mc{B}^{\dg}_{\mc{N}\mc{L}} \n\mc{A}^{\dg}_{\mc{M}\mc{N}}$, we have to show that  $\mc{B}^{\dg}_{\mc{N}\mc{L}} \n\mc{A}^{\dg}_{\mc{M}\mc{N}}$ satisfies Definition \ref{defmpi}, and is shown below. Using Equation \eqref{eq21}, we have
\begin{eqnarray*}
\mc{A}\n\mc{B} \lp\mc{B}^{\dg}_{\mc{N}\mc{L}} \n\mc{A}^{\dg}_{\mc{M}\mc{N}} \m\mc{A} \n\mc{B} 
&=&  \mc{A}\n\mc{B}.
\end{eqnarray*}
Using Lemma \ref{lm3.13}, we get $\mc{B}\lp\mc{B}^{\#}_{\mc{N}\mc{L}}\n\mc{A}^{\dg}_{\mc{M}\mc{N}}\m\mc{A}\n\mc{B}\lp\mc{B}^{\dg}_{\mc{N}\mc{L}} \n\mc{A}^{\#}_{\mc{M}\mc{N}} = \mc{B}\lp\mc{B}^{\#}_{\mc{N}\mc{L}} \n\mc{A}^{\#}_{\mc{M}\mc{N}}$, which on pre-multiplying by $\mc{B}^{\dg}_{\mc{N}\mc{L}}$ gives
$\mc{B}^{\#}_{\mc{N}\mc{L}}\n\mc{A}^{\dg}_{\mc{M}\mc{N}}\m\mc{A}\n\mc{B}\lp\mc{B}^{\dg}_{\mc{N}\mc{L}} \n\mc{A}^{\#}_{\mc{M}\mc{N}} = \mc{B}^{\#}_{\mc{N}\mc{L}} \n \mc{A}^{\#}_{\mc{M}\mc{N}}$. Again, pre-multiplying $(\mc{B}^{\dg}_{\mc{N}\mc{L}})^{\#}_{\mc{L}\mc{N}}$ and post-multiplying $(\mc{A}^{\dg}_{\mc{M}\mc{N}})^{\#}_{\mc{N}\mc{M}}$ yields
\begin{eqnarray*}
 \mc{B}^{\dg}_{\mc{N}\mc{L}}\n\mc{A}^{\dg}_{\mc{M}\mc{N}}\m\mc{A}\n\mc{B}\lp\mc{B}^{\dg}_{\mc{N}\mc{L}} \n\mc{A}^{\dg}_{\mc{M}\mc{N}} &=&\mc{B}^{\dg}_{\mc{N}\mc{L}} \n \mc{A}^{\dg}_{\mc{M}\mc{N}}.
\end{eqnarray*}
Now, with the help of Equation \eqref{eq23}, we get 
\begin{eqnarray*}
(\mc{M}\m\mc{A}\n\mc{B}\lp\mc{B}^{\dg}_{\mc{N}\mc{L}}\n\mc{A}^{\dg}_{\mc{M}\mc{N}})^{H}
&=&\mc{M}\m\mc{A}\n\mc{B}\lp\mc{B}^{\dg}_{\mc{N}\mc{L}}\n\mc{A}^{\dg}_{\mc{M}\mc{N}},
\end{eqnarray*}
and using Equation \eqref{eq21}, we get
\begin{eqnarray*}
(\mc{L}\lp\mc{B}^{\dg}_{\mc{N}\mc{L}}\n\mc{A}^{\dg}_{\mc{M}\mc{N}}\m\mc{A}\n\mc{B})^{H}
&=&\mc{L}\lp\mc{B}^{\dg}_{\mc{N}\mc{L}}\n\mc{A}^{\dg}_{\mc{M}\mc{N}}\m\mc{A}\n\mc{B}.
\end{eqnarray*}
\end{proof}
Next, we provide a simpler characterization than that of Theorem \ref{wrol1} of the reverse-order law for the weighted Moore-Penrose inverse of the Einstein product of two tensors.
\begin{theorem}\label{wrol2}
Let $\mc{A}\in {\C}^{I_{1}\times \cdots\times I_{M}\times J_{1}\times \cdots \times J_{N}}$ and  $\mc{B}\in {\C}^{J_{1}\times \cdots \times J_{N}\times K_{1}\times \cdots\times K_{L}}$. Also, let $\mc{M}\in\mc{C}^{I_{1}\times \cdots\times I_{M}\times I_{1}\times \cdots\times I_{M}}$, $\mc{N}\in \mc{C}^{J_{1}\times \cdots\times J_{N}\times J_{1}\times \cdots\times J_{N}}$ and  $\mc{L}\in \mc{C}^{K_{1}\times \cdots\times K_{L}\times K_{1}\times \cdots\times K_{L}}$ be Hermitian positive definite tensors. Then 
$$(\mc{A}\n\mc{B})^{\dg}_{\mc{M}\mc{L}}=\mc{B}^{\dg}_{\mc{N}\mc{L}}\n\mc{A}^{\dg}_{\mc{M}\mc{N}}$$
if and only if  
\begin{eqnarray*}
(\mc{A}^{\dg}_{\mc{M}\mc{N}} \m\mc{A}\n\mc{B}\lp\mc{B}^{\#}_{\mc{N}\mc{L}})^{\#}_{\mc{N}\mc{N}}&=&\mc{A}^{\dg}_{\mc{M}\mc{N}} \m\mc{A}\n\mc{B}\n\mc{B}^{\#}_{\mc{N}\mc{L}}, \\
(\mc{A}^{\#}_{\mc{M}\mc{N}} \m\mc{A}\n\mc{B}\lp\mc{B}^{\dg}_{\mc{N}\mc{L}})^{\#}_{\mc{N}\mc{N}}&=&\mc{A}^{\#}_{\mc{M}\mc{N}} \m\mc{A}\n\mc{B}\lp\mc{B}^{\dg}_{\mc{N}\mc{L}}.
\end{eqnarray*}
\end{theorem}
In the last two Theorems \ref{wrol1} and  \ref{wrol2}, we need to check two conditions for the reverse-order law, however the next theorem needs only one condition to check.
\begin{theorem}
Let $\mc{A}\in {\C}^{I_{1}\times \cdots\times I_{M}\times J_{1}\times \cdots \times J_{N}}$ and  $\mc{B}\in {\C}^{J_{1}\times \cdots \times J_{N}\times K_{1}\times \cdots\times K_{L}}$. Also, let $\mc{M}\in\mc{C}^{I_{1}\times \cdots\times I_{M}\times I_{1}\times \cdots\times I_{M}}$, $\mc{N}\in \mc{C}^{J_{1}\times \cdots\times J_{N}\times J_{1}\times \cdots\times J_{N}}$ and  $\mc{L}\in \mc{C}^{K_{1}\times \cdots\times K_{L}\times K_{1}\times \cdots\times K_{L}}$ be Hermitian positive definite tensors. Then $(\mc{A}\n\mc{B})^{\dg}_{\mc{M}\mc{L}}=\mc{B}^{\dg}_{\mc{N}\mc{L}}\n\mc{A}^{\dg}_{\mc{M}\mc{N}}$ if and only if
\begin{equation*}\label{eq30}
\mc{A}^{\dg}_{\mc{M}\mc{N}} \m \mc{A} \n \mc{B} \lp \mc{B}^{\#}_{\mc{N}\mc{L}} \n\mc{A}^{\#}_{\mc{M}\mc{N}} \m \mc{A} \n \mc{B} \lp \mc{B}^{\dg}_{\mc{N}\mc{L}} = \mc{B} \lp \mc{B}^{\#}_{\mc{N}\mc{L}} \n\mc{A}^{\#}_{\mc{M}\mc{N}} \m \mc{A}.
\end{equation*}
\end{theorem}

We next present another characterization of the reverse-order law.

\begin{theorem}\label{theorem4}
Let $\mc{A}\in {\C}^{I_{1}\times \cdots\times I_{M}\times J_{1}\times \cdots \times J_{N}}$ and  $\mc{B}\in {\C}^{J_{1}\times \cdots \times J_{N}\times K_{1}\times \cdots\times K_{L}}$. Also, let $\mc{M}\in\mc{C}^{I_{1}\times \cdots\times I_{M}\times I_{1}\times \cdots\times I_{M}}$, $\mc{N}\in \mc{C}^{J_{1}\times \cdots\times J_{N}\times J_{1}\times \cdots\times J_{N}}$ and  $\mc{L}\in \mc{C}^{K_{1}\times \cdots\times K_{L}\times K_{1}\times \cdots\times K_{L}}$ be Hermitian positive definite tensors.  Then $(\mc{A}\n\mc{B})^{\dg}_{\mc{M}\mc{L}} =\mc{B}^{\dg}_{\mc{N}\mc{L}} \n \mc{A}^{\dg}_{\mc{M}\mc{N}}$ if and only if 
\begin{equation*}\label{eq35}
\mc{A}^{\dg}_{\mc{M}\mc{N}} \m \mc{A} \n \mc{B} = \mc{B}\lp(\mc{A}\n\mc{B})^{\dg}_{\mc{M}\mc{L}} \m \mc{A} \n \mc{B}
\end{equation*}
and
\begin{equation*}\label{eq36}
\mc{B} \lp \mc{B}^{\dg}_{\mc{N}\mc{L}} \n \mc{A}^{\#}_{\mc{M}\mc{N}} = \mc{A}^{\#}_{\mc{M}\mc{N}} \m\mc{A}\n\mc{B} \lp (\mc{A} \n \mc{B})^{\dg}_{\mc{M}\mc{L}}.
\end{equation*}
\end{theorem}
Next theorem replaces the first condition of Theorem \ref{wrol1} by a new condition.
\begin{theorem}
Let $\mc{A}\in {\C}^{I_{1}\times \cdots\times I_{M}\times J_{1}\times \cdots \times J_{N}}$ and  $\mc{B}\in {\C}^{J_{1}\times \cdots \times J_{N}\times K_{1}\times \cdots\times K_{L}}$. Also, let $\mc{M}\in\mc{C}^{I_{1}\times \cdots\times I_{M}\times I_{1}\times \cdots\times I_{M}}$, $\mc{N}\in \mc{C}^{J_{1}\times \cdots\times J_{N}\times J_{1}\times \cdots\times J_{N}}$ and  $\mc{L}\in \mc{C}^{K_{1}\times \cdots\times K_{L}\times K_{1}\times \cdots\times K_{L}}$ be Hermitian positive definite tensors.  Then $(\mc{A}\n\mc{B})^{\dg}_{\mc{M}\mc{L}} =\mc{B}^{\dg}_{\mc{N}\mc{L}} \n \mc{A}^{\dg}_{\mc{M}\mc{N}}$ if and only if 
\begin{itemize}
    \item[(i)] $(\mc{L}\lp\mc{B}^{\dg}_{\mc{N}\mc{L}}\n\mc{A}^{\dg}_{\mc{M}\mc{N}}\m\mc{A}\n\mc{B})^{H}=\mc{L}\lp\mc{B}^{\dg}_{\mc{N}\mc{L}}\n\mc{A}^{\dg}_{\mc{M}\mc{N}}\m\mc{A}\n\mc{B}$,
    \item[(ii)] $\mc{B}\lp\mc{B}^{\dg}_{\mc{N}\mc{L}}\n\mc{A}^{\#}_{\mc{M}\mc{N}}\m\mc{A}\n\mc{B}=\mc{A}^{\#}_{\mc{M}\mc{N}}\m\mc{A}\n\mc{B}$.
\end{itemize}
\end{theorem}

Similarly, we replace the second condition of Theorem \ref{wrol1} by a new condition.
\begin{theorem}
Let $\mc{A}\in {\C}^{I_{1}\times \cdots\times I_{M}\times J_{1}\times \cdots \times J_{N}}$ and  $\mc{B}\in {\C}^{J_{1}\times \cdots \times J_{N}\times K_{1}\times \cdots\times K_{L}}$. Also, let $\mc{M}\in\mc{C}^{I_{1}\times \cdots\times I_{M}\times I_{1}\times \cdots\times I_{M}}$, $\mc{N}\in \mc{C}^{J_{1}\times \cdots\times J_{N}\times J_{1}\times \cdots\times J_{N}}$ and  $\mc{L}\in \mc{C}^{K_{1}\times \cdots\times K_{L}\times K_{1}\times \cdots\times K_{L}}$ be Hermitian positive definite tensors.  Then $(\mc{A}\n\mc{B})^{\dg}_{\mc{M}\mc{L}} =\mc{B}^{\dg}_{\mc{N}\mc{L}} \n \mc{A}^{\dg}_{\mc{M}\mc{N}}$ if and only if 
\begin{itemize}
    \item[(i)] $\mc{A}^{\dg}_{\mc{M}\mc{N}}\m\mc{A}\n\mc{B}\lp\mc{B}^{\#}_{\mc{N}\mc{L}}\n\mc{A}^{\#}_{\mc{M}\mc{N}}=\mc{B}\lp\mc{B}^{\#}_{\mc{N}\mc{L}}\n\mc{A}^{\#}_{\mc{M}\mc{N}}$,
    \item[(ii)] $(\mc{M}\m\mc{A}\n\mc{B}\lp\mc{B}^{\dg}_{\mc{N}\mc{L}}\n\mc{A}^{\dg}_{\mc{M}\mc{N}})^{H}=\mc{M}\m\mc{A}\n\mc{B}\lp\mc{B}^{\dg}_{\mc{N}\mc{L}}\n\mc{A}^{\dg}_{\mc{M}\mc{N}}$.
\end{itemize}
\end{theorem}

It is interesting to note that the conditions   $\mc{M}\m\mc{A}\n\mc{B}\lp\mc{B}^{\dg}_{\mc{N}\mc{L}}\n\mc{A}^{\dg}_{\mc{M}\mc{N}}$ and $\mc{N}\lp\mc{B}^{\dg}_{\mc{N}\mc{L}}*_{N}\mc{A}^{\dg}_{\mc{M}\mc{N}}\m\mc{A}\n\mc{B}$ being Hermitian are weaker conditions than are $\mc{B}\m\mc{B}^{\dg}_{\mc{N}\mc{L}}\n\mc{A}^{\#}_{\mc{M}\mc{N}}\m\mc{A}\n\mc{B}=\mc{A}^{\#}_{\mc{M}\mc{N}}\m\mc{A}\n\mc{B}$ and $\mc{A}^{\dg}_{\mc{M}\mc{N}}\m\mc{A}\n\mc{B}\m\mc{B}^{\#}_{\mc{N}\mc{L}}\n\mc{A}^{\#}_{\mc{M}\mc{N}}=\mc{B}\m\mc{B}^{\#}_{\mc{N}\mc{L}}\n\mc{A}^{\#}_{\mc{M}\mc{N}}$. The last result of this paper shows that the reverse-order law is a sufficient condition for the commutativity of $\mc{A}^{\dg}_{\mc{M}\mc{N}}\m\mc{A}$ and $\mc{B}\lp\mc{B}^{\dg}_{\mc{N}\mc{L}}$.

\begin{theorem}
Let $\mc{A}\in {\C}^{I_{1}\times \cdots\times I_{M}\times J_{1}\times \cdots \times J_{N}}$ and  $\mc{B}\in {\C}^{J_{1}\times \cdots \times J_{N}\times K_{1}\times \cdots\times K_{L}}$. Also, let $\mc{M},~\mc{N},$ $\mc{L}$ be Hermitian positive definite tensors in $\mc{C}^{I_{1}\times I_{2}\times \cdots\times I_{M}\times I_{1}\times I_{2}\times \cdots\times I_{M}}$, $\mc{C}^{J_{1}\times J_{2}\times \cdots\times J_{N}\times J_{1}\times J_{2}\times \cdots\times J_{N}}$ and $\mc{C}^{K_{1}\times \cdots\times K_{L}\times K_{1}\times \cdots\times K_{L}}$, respectively.  If $(\mc{A}\n\mc{B})^{\dg}_{\mc{M}\mc{L}} =\mc{B}^{\dg}_{\mc{N}\mc{L}} \n \mc{A}^{\dg}_{\mc{M}\mc{N}}$, then  $\mc{A}^{\dg}_{\mc{M}\mc{N}} \m \mc{A} $ 
and $\mc{B} \lp \mc{B}^{\dg}_{\mc{N}\mc{L}} $ commute.
\end{theorem}


\section*{References}

\bibliographystyle{amsplain}

\end{document}